\documentclass[a4paper,12pt,reqno]{amsart}
\usepackage{amsfonts, color}
\usepackage{amsmath}
\usepackage{amssymb}
\usepackage[a4paper]{geometry}
\usepackage{mathrsfs}
\usepackage[colorlinks]{hyperref}
\usepackage{hyperref}
\renewcommand\eqref[1]{(\ref{#1})}

\makeatletter
\newcommand*{\mint}[1]{%
  \mint@l{#1}{}%
}
\newcommand*{\mint@l}[2]{%
  \@ifnextchar\limits{%
    \mint@l{#1}%
  }{%
    \@ifnextchar\nolimits{%
      \mint@l{#1}%
    }{%
      \@ifnextchar\displaylimits{%
        \mint@l{#1}%
      }{%
        \mint@s{#2}{#1}%
      }%
    }%
  }%
}
\newcommand*{\mint@s}[2]{%
  \@ifnextchar_{%
    \mint@sub{#1}{#2}%
  }{%
    \@ifnextchar^{%
      \mint@sup{#1}{#2}%
    }{%
      \mint@{#1}{#2}{}{}%
    }%
  }%
}
\def\mint@sub#1#2_#3{%
  \@ifnextchar^{%
    \mint@sub@sup{#1}{#2}{#3}%
  }{%
    \mint@{#1}{#2}{#3}{}%
  }%
}
\def\mint@sup#1#2^#3{%
  \@ifnextchar_{%
    \mint@sup@sub{#1}{#2}{#3}%
  }{%
    \mint@{#1}{#2}{}{#3}%
  }%
}
\def\mint@sub@sup#1#2#3^#4{%
  \mint@{#1}{#2}{#3}{#4}%
}
\def\mint@sup@sub#1#2#3_#4{%
  \mint@{#1}{#2}{#4}{#3}%
}
\newcommand*{\mint@}[4]{%
  \mathop{}%
  \mkern-\thinmuskip
  \mathchoice{%
    \mint@@{#1}{#2}{#3}{#4}%
        \displaystyle\textstyle\scriptstyle
  }{%
    \mint@@{#1}{#2}{#3}{#4}%
        \textstyle\scriptstyle\scriptstyle
  }{%
    \mint@@{#1}{#2}{#3}{#4}%
        \scriptstyle\scriptscriptstyle\scriptscriptstyle
  }{%
    \mint@@{#1}{#2}{#3}{#4}%
        \scriptscriptstyle\scriptscriptstyle\scriptscriptstyle
  }%
  \mkern-\thinmuskip
  \int#1%
  \ifx\\#3\\\else_{#3}\fi
  \ifx\\#4\\\else^{#4}\fi
}
\newcommand*{\mint@@}[7]{%
  \begingroup
    \sbox0{$#5\int\m@th$}%
    \sbox2{$#5\int_{}\m@th$}%
    \dimen2=\wd0 %
    \let\mint@limits=#1\relax
    \ifx\mint@limits\relax
      \sbox4{$#5\int_{\kern1sp}^{\kern1sp}\m@th$}%
      \ifdim\wd4>\wd2 %
        \let\mint@limits=\nolimits
      \else
        \let\mint@limits=\limits
      \fi
    \fi
    \ifx\mint@limits\displaylimits
      \ifx#5\displaystyle
        \let\mint@limits=\limits
      \fi
    \fi
    \ifx\mint@limits\limits
      \sbox0{$#7#3\m@th$}%
      \sbox2{$#7#4\m@th$}%
      \ifdim\wd0>\dimen2 %
        \dimen2=\wd0 %
      \fi
      \ifdim\wd2>\dimen2 %
        \dimen2=\wd2 %
      \fi
    \fi
    \rlap{%
      $#5%
        \vcenter{%
          \hbox to\dimen2{%
            \hss
            $#6{#2}\m@th$%
            \hss
          }%
        }%
      $%
    }%
  \endgroup
}
\numberwithin{equation}{section}
\theoremstyle{plain}
\newtheorem{thm}{Theorem}[section]

\newtheorem{cor}[thm]{Corollary}

\theoremstyle{definition}
\newtheorem{defn}[thm]{Definition}
\newtheorem{rem}[thm]{Remark}

\title[Wave equation for Sturm-Liouville operator with singular potentials]{Wave equation for Sturm-Liouville operator with singular potentials}
\author[M. Ruzhansky]{Michael Ruzhansky}
\address{
  Michael Ruzhansky:
  \endgraf
  Department of Mathematics: Analysis, Logic and Discrete Mathematics
  \endgraf
  Ghent University, Belgium
  \endgraf
 and
  \endgraf
  School of Mathematical Sciences
  \endgraf
  Queen Mary University of London
  \endgraf
  United Kingdom
  \endgraf
  {\it E-mail address} {\rm michael.ruzhansky@ugent.be}
  }

\author[S. Shaimardan]{Serikbol Shaimardan}
\address{
	Serikbol Shaimardan:
  \endgraf
  Department of Mathematics: Analysis, Logic and Discrete Mathematics
  \endgraf
  Ghent University, Belgium
  \endgraf
 and 
  \endgraf
     L.N. Gumilyov Eurasian National University
    \endgraf
    Nur-Sultan, Kazakhstan
  	\endgraf
	{\it E-mail address} {\rm shaimardan.serik@gmail.com}
		}
	
\author[A. Yeskermessuly]{Alibek Yeskermessuly}
\address{
  Alibek Yeskermessuly:
    \endgraf
  Department of Mathematics: Analysis, Logic and Discrete Mathematics
  \endgraf
  Ghent University, Belgium
  \endgraf
 and 
   \endgraf
  Altynsarin Arkalyk Pedagogical Institute,  \\ 
  \endgraf
  Arkalyk, Kazakhstan
  \endgraf
  {\it E-mail address} {\rm alibek.yeskermessuly@gmail.com}}

\begin{document}

\thanks{The authors are supported by the FWO Odysseus 1 grant G.0H94.18N: Analysis and Partial Differential Equations and by the Methusalem programme of the Ghent University Special Research Fund (BOF) (Grant number 01M01021). Michael Ruzhansky is also supported by EPSRC grant EP/R003025/2, and the second and the third authors by the international internship program \textquotedblleft Bolashak\textquotedblright \,of the Republic of Kazakhstan. }

%
%

%
%

%
\maketitle              

\begin{abstract}
The paper is denoted to the initial-boundary value problem for the wave equation with the Sturm-Liouville operator with irregular (distributive) potentials. To obtain a solution to the equation, the separation method and asymptotics of the eigenvalues and eigenfunctions of the Sturm-Liouville operator are used. Homogeneous and inhomogeneous cases of the equation are considered. Next, existence, uniqueness, and consistency theorems for a very weak solution of the wave equation with singular coefficients are proved.

\end{abstract}

\section{Introduction}

The purpose of this paper is to establish the well-posedness results for the wave equation for the Sturm-Liouville operator with singular potentials. To treat problems with strong singularities, the notion of very weak solutions was introduced in \cite{Gar-Ruz}. The idea is that when a product of different terms appear in the equation, it is no longer well-defined in spaces of distributions; therefore, another interpretation of the well-posedness of the equation is needed.

In the work \cite{Lan-Ham}, very weak solutions of the wave equation for the Landau Hamiltonian with an irregular electromagnetic field were obtained. Further development of this notion for different types of problems was continued in works \cite{ARST1}, \cite{ARST2}, \cite{ARST3}, \cite{CRT1}, \cite{CRT2}, \cite{R-Y}. Compared to Colombeau algebras techniques, the notion of very weak solution is more flexible: it can be easily adapted to the equation under consideration, it does not require the verification (and validity) of algebra properties, and it is consistent with the notions of distributional, weak, and classical solutions.

It is well known that the wave equation can be easily reduced to ordinary linear equations by the method of \textquotedblleft separation of variables\textquotedblright \, (see, for example, \cite{Separ}). 
To obtain our main results, we present some preliminary information about the Sturm-Liouville operator with singular potentials. By Savchuk and Shkalikov, in \cite{Sav-Shk}, eigenvalues and eigenfunctions for the Sturm-Liouville operator with singular potentials were obtained. Also in \cite{N-zSk}, \cite{Savch}, \cite{Sav-Shk2}, \cite{SV} one studied the Sturm-Liouville operator with potential-distributions. In order to set up the machinery of very weak solutions, here we are rather interested in estimates for solutions of more regular problems, keeping track of a more explicit dependence on the regularisation parameter.

More specifically, we consider the Sturm-Liouville operator $\mathcal{L}$ generated on the interval (0,1) by the differential expression 
\begin{equation}\label{St-L}
    \mathcal{L}y:=-\frac{d^2}{dx^2}y+q(x)y,
\end{equation}
with the boundary conditions
\begin{equation}\label{Dirihle}
    y(0)=y(1)=0. 
\end{equation}
The potential $q$ is defined as
\begin{equation}\label{con-q}
    q(x)=\nu'(x), \qquad \nu\in L^2(0,1).
\end{equation}

We introduce the quasi-derivative in the following form
$$y^{[1]}(x)=y'(x)-\nu(x)y(x),$$
so that the eigenvalue equation $\mathcal{L}y=\lambda y$ transforms to the system
$$\left(\begin{array}{c}
     \phi_1  \\
     \phi_2 
\end{array}\right)'=\left(\begin{array}{cc}
    \nu & 1 \\
    -\lambda-\nu^2 & -\nu
\end{array}\right)\left(\begin{array}{c}
     \phi_1  \\
     \phi_2
\end{array}\right),\quad \phi_1(x)=y(x),\,\,\phi_2(x)=y^{[1]}(x).$$

We make the substitution
$$\phi_1(x)=r(x)\sin \theta(x),\qquad \phi_2(x)=\lambda^\frac{1}{2}r(x)\cos \theta(x),$$
which is a modification of the Prufer substitution (\cite{Ince}), where
\begin{equation}\label{theta}
\theta'(x,\lambda)=\lambda^\frac{1}{2}+\lambda^{-\frac{1}{2}}\nu^2(x)\sin^2 \theta(x,\lambda)+\nu(x)\sin 2\theta(x,\lambda),
\end{equation}
\begin{equation}\label{r}
    r'(x,\lambda)=-r(x,\lambda)\left[\frac{1}{2}\nu^2(x)\lambda^{-\frac{1}{2}}\sin 2\theta(x,\lambda)+\nu(x)\cos 2\theta(x,\lambda)\right].
\end{equation}

The solution of the equation \eqref{theta} will be sought in the form $\theta(x,\lambda)=\lambda^\frac{1}{2} x+\eta(x,\lambda), $ where
$$\eta(x,\lambda)=\lambda^{-\frac{1}{2}}\int\limits_0^x \nu^2(s)\sin^2\theta(s,\lambda)ds+\int\limits_0^x \nu(s)\sin(2\lambda^\frac{1}{2}s+2\eta(s,\lambda))ds.$$
Using the method of successive approximations, it is easy to show that this equation has a solution that is uniformly bounded for $0\leq x\leq 1$ and $\lambda\geq 1$. Since $|\nu|^2\in L^1(0,1)$, by virtue of the Riemann-Lebesgue lemma, $\eta(x,\lambda)=o(1)$ as $\lambda \to \infty$. Therefore,
$$\theta(x,\lambda)=\lambda^\frac{1}{2}x+o(1),$$
moreover $\theta(0,\lambda)=0.$

Using the Riemann-Lebesgue lemma again, from equation \eqref{r} we find
$$r(x,\lambda)=\exp{\left(-\int\limits_0^x \nu(s)\cos2\theta(s,\lambda)ds-\frac{1}{2}\lambda^{-\frac{1}{2}}\int\limits_0^x \nu^2(s)\sin 2\theta(s,\lambda)ds\right)}.$$

Using the boundary conditions \eqref{Dirihle} we obtain
$$ \phi_1(1,\lambda)=r(1,\lambda)\sin\theta(1,\lambda)=0,\,\, r(1,\lambda)\neq 0,\,\,\theta(1,\lambda)=\pi n.$$
Then the eigenvalues of the Sturm-Liouville operator $\mathcal{L}$ generated on the interval (0,1) by the differential expression \eqref{St-L} with the boundary conditions \eqref{Dirihle} are given by
\begin{equation}\label{e-val}
    \lambda_n=(\pi n)^2(1+o(n^{-1})),\qquad n=1,2,...,
\end{equation}
and the corresponding eigenfunctions are
\begin{equation}\label{sol-SL}
    \Tilde{\phi}_n(x)=r_n(x)\sin(\sqrt{\lambda_n}x +\eta_n(x)).
\end{equation}

The first derivatives of $\Tilde{\phi_n}$ are given by the formulas
\begin{equation}\label{phi-der}
    \Tilde{\phi}'_n(x)=\sqrt{\lambda_n}r_n(x)\cos(\theta_n(x))+\nu(x)\Tilde{\phi}_n(x).
\end{equation}

According Theorem 2 in \cite{Savch} we have
\begin{equation}\label{phi-sav}
\Tilde{\phi}_n(x)=\sin{\sqrt{\lambda_n}x}+\psi_n(x), \quad n=1,2,...,\quad \sum\limits_{n=1}^\infty \|\psi_n\|^2\leq C \int\limits_0^1|\nu(x)|^2dx.
\end{equation}

On the other hand, we can estimate the $\|\Tilde{\phi}_n\|_{L^2}$ using the formula \eqref{sol-SL} as follows
\begin{eqnarray}\label{est-high}
\|\Tilde{\phi}_n\|^2_{L^2}&=&\int\limits_0^1\left|r_n(x)\sin\left(\lambda_n^{\frac{1}{2}}x+\eta_n(x)\right)\right|^2dx\leq \int\limits_0^1\left|r_n(x)\right|^2dx\nonumber\\
&\leq& \int\limits_0^1\left|\exp{\left(-\int\limits_0^x \nu(s)\cos{2\theta_n(s)}ds-\frac{1}{2}\frac{1}{\sqrt{\lambda_n}}\int\limits_0^x\nu^2(s)\sin{2\theta_n(s)}ds\right)}\right|^2dx\nonumber\\
&\lesssim& \int\limits_0^1\exp{\left(2\int\limits_0^x|\nu(s)|ds+\frac{1}{\sqrt{\lambda_n}}\int\limits_0^x|\nu^2(s)|ds\right)}dx\nonumber\\
&\lesssim& \exp{\left(\|\nu\|_{L^2}+\lambda^{-\frac{1}{2}}\|\nu\|^2_{L^2}\right)}<\infty.
\end{eqnarray}

Also, according Theorem 4 in \cite{Sav-Shk}, we have
\begin{equation}\label{est_low}
  \Tilde{\phi}_n(x)=\sin(\pi nx)+o(1)  
\end{equation}
for sufficiently large $n$, it means that there exist some $C_0>0$, such that
\begin{equation}\label{low-est}
C_0<\|\Tilde{\phi}_n\|_{L^2}<\infty.
\end{equation}
Since the eigenfunctions of the Sturm-Liouville operator form an orthogonal basis in $L^2(0,1)$, we normalize them for further use  
\begin{equation}\label{norm-phi}
  \phi_n(x)=\frac{\Tilde{\phi}_n(x)}{\sqrt{\langle \Tilde{\phi}_n,\Tilde{\phi}_n}\rangle}=\frac{\Tilde{\phi}_n(x)}{\|\Tilde{\phi}_n\|_{L^2}}.  
\end{equation}

\section{Main results}

We consider the wave equation
\begin{equation}\label{C.p1}
         \partial^2_t u(t,x)+\mathcal{L} u(t,x)=0,\qquad (t,x)\in [0,T]\times (0,1),
\end{equation}
with initial conditions
\begin{equation}\label{C.p2} \left\{\begin{array}{l}u(0,x)=u_0(x),\,\,\, x\in (0,1), \\
\partial_t u(0,x)=u_1(x), \,\,\, x\in (0,1),\end{array}\right.\end{equation}
and with Dirichlet boundary conditions
\begin{equation}\label{C.p3}
u(t,0)=0=u(t,1),\qquad t\in [0,T],
\end{equation}
where  $\mathcal{L}$ is defined by

\begin{equation}\label{1}
    \mathcal{L} u(t,x):=-\partial^2_x u(t,x)+q(x)u(t,x),\qquad x\in(0,1),
\end{equation}
and $q$ is defined as in \eqref{con-q}. 

In our results below, concerning the initial/boundary problem \eqref{C.p1}-\eqref{C.p3}, as the preliminary step we first carry out the analysis in the strictly regular case for bounded $q \in L^\infty(0,1)$. In this case, we obtain the well-posedness in the Sobolev spaces $W^k_\mathcal{L}$ associated to the operator $\mathcal{L}$: we define the Sobolev spaces $W^k_\mathcal{L}$ associated to $\mathcal{L}$, for any $k \in \mathbb{R}$, as the space  
$$W^k_\mathcal{L}:=\left\{f\in \mathcal{D}'_\mathcal{L}(0,1):\,\mathcal{L}^{k/2}f\in L^2(0,1)\right\},$$
with the norm $\|f\|_{W^k_\mathcal{L}}:=\|\mathcal{L}^{k/2}f\|_{L^2}$. The global space of distributions $\mathcal{D}'_\mathcal{L}(0,1)$ is defined as follows.

The space $C^\infty_\mathcal{L}(0,1):=\mathrm{Dom}(\mathcal{L}^\infty)$ is called the space of test functions for $\mathcal{L}$, where we define 
$$\mathrm{Dom}(\mathcal{L}^\infty):=\bigcap\limits_{m=1}^\infty \mathrm{Dom}(\mathcal{L}^m),$$
where $\mathrm{Dom}(\mathcal{L}^m)$ is the domain of the operator $\mathcal{L}^m$, in turn defined as
$$\mathrm{Dom}(\mathcal{L}^m):=\left\{f\in L^2(0,1): \mathcal{L}^j f\in \mathrm{Dom}(\mathcal{L}),\,\, j=0,1,2,...,m-1\right\}.$$
The Fréchet topology of $C^\infty_\mathcal{L}(0,1)$ is given by the family of norms 
\begin{equation}\label{frechet}
    \|\phi\|_{C^m_\mathcal{L}}:=\max\limits_{j\leq m}\|\mathcal{L}^j\phi\|_{L^2(0,1)},\quad m\in \mathbb{N}_0,\,\, \phi\in C^\infty_\mathcal{L}(0,1).
\end{equation}
The space of $\mathcal{L}$-distributions
$$\mathcal{D}'_\mathcal{L}(0,1):=\mathbf{L}\left(C^\infty_\mathcal{L}(0,1),\mathbb{C}\right)$$
is the space of all linear continuous functionals on $C^\infty_\mathcal{L}(0,1)$. For $\omega \in \mathcal{D}'_\mathcal{L}(0,1)$ and $\phi\in C^\infty_\mathcal{L}(0,1)$, we shall write 
$$\omega(\phi)=\langle \omega, \phi\rangle.$$
For any $\psi \in C^\infty_\mathcal{L}(0,1)$, the functional 
$$C^\infty_\mathcal{L}(0,1)\ni \phi \mapsto \int\limits_0^1 \psi(x)\phi(x)dx$$
is an $\mathcal{L}$-distribution, which gives an embedding $\psi \in C^\infty_\mathcal{L}(0,1)\hookrightarrow \mathcal{D}'_\mathcal{L}(0,1)$.

We introduce the spaces $C^j([0,T],W^k_\mathcal{L}(0,1))$ given by the family of norms
\begin{equation}
    \|f\|_{C^n([0,T],W^k_\mathcal{L}(0,1))}=\max\limits_{0\leq t\leq T}\sum\limits_{j=0}^n\left\|\partial^j_t f(t,\cdot)\right\|_{W^k_\mathcal{L}},
\end{equation}
where $k\in \mathbb{R}, \, f\in C^j([0,T],W^k_\mathcal{L}(0,1)).$
 
\begin{thm}\label{th1}
Assume that $q \in L^\infty(0,1)$. For any $k\in \mathbb{R}$, if the initial data satisfy $(u_0,\, u_1) \in W^{1+k}_\mathcal{L}\times W^k_\mathcal{L}$ then the wave equation \eqref{C.p1} with the initial/boundary problem  \eqref{C.p2}-\eqref{C.p3}  has unique solution $u\in C([0,T], W^{1+k}_\mathcal{L})\cap C^1([0,T], W^{k}_\mathcal{L})$. We also have the following estimates:
\begin{equation}\label{est1}
    \|u(t,\cdot)\|^2_{L^2}\lesssim \|u_0\|^2_{L^2}+\|u_1\|^2_{{W^{-1}_\mathcal{L}}},
\end{equation}
\begin{equation}\label{est2}
\|\partial_t u(t,\cdot)\|^2_{L^2}\lesssim \| u_0\|^2_{_{W^1_\mathcal{L}}}+\|u_1\|^2_{L^2},
\end{equation}
\begin{eqnarray}\label{est3}
\|\partial_x u(t,\cdot)\|^2_{L^2}
&\lesssim& \left(1+\|\nu\|^2_{L^2}\right)\left(\|u_0\|^2_{W^1_\mathcal{L}}+\|u_1\|^2_{L^2}\right)\nonumber\\
&+&\|\nu\|^2_{L^\infty}\left(\|u_0\|^2_{L^2}+\|u_1\|^2_{W^{-1}_{\mathcal{L}}}\right),
\end{eqnarray}
\begin{eqnarray}\label{est4}
\|\partial^2_xu(t,\cdot)\|^2_{L^2} &\lesssim& \|q\|^2_{L^\infty} \left(\|u_0\|^2_{L^2}+\|u_1\|^2_{W^{-1}_\mathcal{L}}\right) +\|u_0\|^2_{W^2_\mathcal{L}}+\|u_1\|^2_{W^1_\mathcal{L}},
\end{eqnarray}
\begin{equation}\label{est5}
    \|u(t,\cdot)\|^2_{W^k_\mathcal{L}} \lesssim \|u_0\|^2_{W^k_\mathcal{L}}+\|u_1\|^2_{W^{k-1}_\mathcal{L}},
\end{equation}
where the constants in these inequalities are independent of $u_0$, $u_1$ and $q$.
\end{thm} 

We note that $q\in L^\infty(0,1)$ implies that $\nu \in L^\infty(0,1)$ and hence $\nu\in L^2(0,1)$, so that the formulas in introduction hold true.

\begin{proof}
We apply the technique of the separation of variables (see, e.g. \cite{Separ}). This method involves finding a solution of a certain form. In particular, we are looking for a solution of the form
$$u(t,x)=T(t)X(x),$$
for functions $T(t)$, $X(x)$ to be determined. Suppose we can find a solution of \eqref{C.p1} of this form. Plugging a function $u(t,x)=T(t)X(x)$ into the wave equation, we arrive at the equation
$$T''(t)X(x)-T(t)X''(x)+q(x)T(t)X(x)=0.$$
Dividing this equation by $T(t)X(x)$, we have
\begin{equation}\label{3-1}
    \frac{T''(t)}{T(t)}=\frac{X''(x)-q(x)X(x)}{X(x)}=-\lambda,
\end{equation}
for some constant $\lambda$. Therefore, if there exists a solution $u(t,x) = T(t)X(x)$ of the wave equation, then $T(t)$ and $X(x)$ must satisfy the equations
$$\frac{T''(t)}{T(t)}=-\lambda,$$
$$\frac{X''(x)-q(x)X(x)}{X(x)}=-\lambda,$$
for some constant $\lambda$. In addition, in order for $u$ to satisfy the boundary conditions \eqref{C.p3}, we need our function $X$ to satisfy the boundary conditions \eqref{Dirihle}. That is, we need to find function $X$ and scalar $\lambda$, such that
\begin{equation}\label{4}
    -X''(x)+q(x)X(x)=\lambda X(x),
\end{equation}
\begin{equation}\label{5}
    X(0)=X(1)=0.
\end{equation}

The equation \eqref{4} with the boundary condition \eqref{5} has the eigenvalues of the form \eqref{e-val} with the corresponding eigenfunctions of the form \eqref{sol-SL} of the Sturm-Liouville operator $\mathcal{L}$ generated by the differential expression \eqref{St-L}.

Further, we solve the left hand side of the equation \eqref{3-1} with respect to the independent variable $t$, that is,
\begin{equation}\label{3}
        T''(t)=-\lambda T(t),  \qquad t\in [0,T].
\end{equation}

It is well known (\cite{Separ}) that the solution of the equation \eqref{3} with the initial conditions \eqref{C.p2} is
$$T_n(t)=A_n \cos \left(\sqrt{\lambda_n}t\right)+\frac{1}{\sqrt{\lambda_n}}B_n \sin\left(\sqrt{\lambda_n}t\right),$$
where
$$A_n=\int\limits_0^1u_0(x)\phi_n(x)dx, \qquad B_n=\int\limits_0^1 u_1(x)\phi_n(x)dx.$$

Thus, the solution of the wave equation \eqref{C.p1}  with the initial/boundary problems \eqref{C.p2}-\eqref{C.p3} has the form
\begin{equation}\label{23}
    u(t,x)=\sum\limits_{n=1}^\infty \left[A_n \cos \left(\sqrt{\lambda_n}t\right)+\frac{1}{\sqrt{\lambda_n}}B_n\sin\sqrt{\lambda_n} t\right]\phi_n(x). 
\end{equation}

Further we will prove that $u\in C^2([0,T],L^2(0,1))$. By using the Cauchy-Schwarz inequality and fixed $t$, we can deduce that

\begin{eqnarray}\label{25}
\|u(t, \cdot)\|^2_{L^2}&=&\int\limits_0^1|u(t,x)|^2dx \nonumber\\
&=&\int\limits_0^1\left|\sum\limits_{n=1}^\infty\left[A_n \cos \sqrt{\lambda_n} t+\frac{1}{\sqrt{\lambda_n}} B_n\sin\sqrt{\lambda_n} t\right]\phi_n(x)\right|^2dx\nonumber\\
&\lesssim& \int\limits_0^1\sum\limits_{n=1}^\infty\left|A_n \cos\sqrt{\lambda_n} t+\frac{1}{\sqrt{\lambda_n} }B_n\sin\sqrt{\lambda_n} t\right|^2|\phi_n(x)|^2dx\nonumber\\
&\leq& \int\limits_0^1\sum\limits_{n=1}^\infty\left(|A_n||\phi_n(x)|+\frac{1}{\sqrt{\lambda_n}}|B_n||\phi_n(x)|\right)^2 dx\nonumber\\
&\lesssim& \sum\limits_{n=1}^\infty\left(\int\limits_0^1|A_n|^2|\phi_n(x)|^2dx+\int\limits_0^1\left|\frac{B_n}{\sqrt{\lambda_n}}\right|^2|\phi_n(x)|^2dx\right). 
\end{eqnarray}

By using the Parseval identity,
we get
$$
\sum\limits_{n=1}^\infty \int\limits_0^1|A_n|^2|\phi_n(x)|^2dx= \sum\limits_{n=1}^\infty |A_n|^2=\int\limits_0^1|u_0(x)|^2dx=\|u_0\|^2_{L^2}.
$$

For the second term in \eqref{25}, using the properties of the eigenvalues of the operator $\mathcal{L}$ and the Parseval identity we obtain the following estimate 
\begin{eqnarray}\label{W-1}
\sum\limits_{n=1}^\infty \int\limits_0^1\left|\frac{B_n}{\sqrt{\lambda_n}}\right|^2|\phi_n(x)|^2dx&=& \sum\limits_{n=1}^\infty \left|\frac{B_n}{\sqrt{\lambda_n}}\right|^2=\sum\limits_{n=1}^\infty \left|\int\limits_0^1\frac{1}{\sqrt{\lambda_n}}u_1(x)\phi_n(x)dx\right|^2\nonumber\\
&=&\sum\limits_{n=1}^\infty \left|\int\limits_0^1\mathcal{L}^{-\frac{1}{2}}u_1(x)\phi_n(x)dx\right|^2\nonumber\\
&=&\sum\limits_{n=1}^\infty \left|\mathcal{L}^{-\frac{1}{2}}u_{1,n}\right|^2=\|\mathcal{L}^{-\frac{1}{2}}u_{1}\|^2_{L_2}=\|u_1\|^2_{W^{-1}_\mathcal{L}}.
\end{eqnarray}

Therefore, we obtained
$$
\|u(t,\cdot)\|^2_{L^2}\lesssim \|u_0\|^2_{L^2}+\|u_1\|^2_{W^{-1}_\mathcal{L}}.
$$

Now, let us estimate
\begin{eqnarray}\label{t26}
\|\partial_t u(t,\cdot)\|^2&=&\int\limits_0^1|\partial_tu(t,x)|^2dt\nonumber\\
&=&\int\limits_0^1\left|\sum\limits_{n=1}^\infty\left[-\sqrt{\lambda_n}A_n\sin\left(\sqrt{\lambda_n}t\right)+\frac{1}{\sqrt{\lambda_n}}\sqrt{\lambda_n}B_n\cos \sqrt{\lambda_n} t\right]\phi_n(x)\right|^2dx \nonumber\\ &\lesssim& \int\limits_0^1\left(\sum\limits_{n=1}^\infty|\sqrt{\lambda_n} A_n |^2+\sum\limits_{n=1}^\infty|B_n|^2\right)|\phi_n(x)|^2dx\nonumber\\
&=&\sum\limits_{n=1}^\infty|\sqrt{\lambda_n} A_n |^2+\sum\limits_{n=1}^\infty|B_n|^2.
\end{eqnarray}
The second term of \eqref{t26} gives the norm of $\|u_1\|^2_{L^2}$ by the Parseval identity. Since $\lambda_n$ are eigenvalues and $\phi_n(x)$ are eigenfunctions of the operator $\mathcal{L}$, we obtain 
\begin{eqnarray}\label{21-1}\sum\limits_{n=1}^\infty|\sqrt{\lambda_n}A_n|^2&=& \sum\limits_{n=1}^\infty\left|\sqrt{\lambda_n}\int\limits_0^1 u_0(x)\phi_n(x)dx\right|^2 =\sum\limits_{n=1}^\infty\left| \int\limits_0^1 \sqrt{\lambda_n}u_0(x)\phi_n(x)dx\right|^2\nonumber\\
&=&\sum\limits_{n=1}^\infty\left| \int\limits_0^1 \mathcal{L}^\frac{1}{2}u_0(x)\phi_n(x)dx\right|^2.
      \end{eqnarray}
It follows from Parseval's identity that
$$\sum\limits_{n=1}^\infty\left| \int\limits_0^1 \mathcal{L}^\frac{1}{2}u_0(x)\phi_n(x)dx\right|^2=\|\mathcal{L}^\frac{1}{2}u_0\|^2_{L^2}=\|u_0\|^2_{W^1_\mathcal{L}}.$$
Thus, 
$$\|\partial_t u(t,\cdot)\|^2_{L^2}\lesssim \|u_0\|^2_{W^1_\mathcal{L}}+\|u_1\|^2_{L^2}.$$

Let us estimate the norm of $\partial_x u(t,\cdot)$ in $L^2$ by using \eqref{phi-der} and \eqref{norm-phi} for $\phi'_n$:
\begin{eqnarray*}
\|\partial_x u(t,\cdot)\|^2_{L^2}&=&\int\limits_0^1|\partial_xu(t,x)|^2dt\\
&=&\int\limits_0^1\left|\sum\limits_{n=1}^\infty\left[A_n\cos\left(\sqrt{\lambda_n}t\right)+\frac{1}{\sqrt{\lambda_n}}B_n\sin \left(\sqrt{\lambda_n}t\right)\right]\phi'_n(x)\right|^2dx\\
&=& \int\limits_0^1\left|\sum\limits_{n=1}^\infty\left[A_n\cos \left(\sqrt{\lambda_n}t\right)+\frac{1}{\sqrt{\lambda_n}}B_n\sin \left(\sqrt{\lambda_n}t\right)\right]\right.\\
&\times&\left.\left(\frac{\sqrt{\lambda_n}r_n(x) \cos\theta_n(x)}{\|\Tilde{\phi}_n\|_{L^2}}+\nu(x)\phi_n(x)\right)\right|^2dx,
\end{eqnarray*}
according \eqref{est-high} and \eqref{est_low}, there exist some $C_0>0$, such that $C_0<\|\Tilde{\phi}_n\|_{L^2}<\infty$, so that
\begin{eqnarray*}
\|\partial_x u(t,\cdot)\|^2_{L^2}&\lesssim&\sum\limits_{n=1}^\infty\left(\left(\left|\sqrt{\lambda_n}A_n \right|^2+|B_n|^2\right)\int\limits_0^1|r_n(x)|^2dx\right)\\
&+&\sum\limits_{n=1}^\infty\left(\left(|A_n|^2+\left|\frac{1}{\sqrt{\lambda_n}}B_n\right|^2\right)\int\limits_0^1|\nu(x)\phi_n(x)|^2dx\right).
\end{eqnarray*}
Here for $r_n(x)$ according Theorem 2 in \cite{Savch}, we have
$$r_n(x)=1+\rho_n(x),\quad \|\rho_n\|_{L^2}\lesssim \|\nu\|_{L^2},$$
where the constant is independent of $\nu$ and $n$. Therefore,
$$\int\limits_0^1|r_n(x)|^2dx\lesssim 1+\|\nu\|^2_{L^2}.$$
For second term we obtain
$$\int\limits_0^1|\nu(x)\phi_n(x)|^2dx\leq \|\nu\|^2_{L^\infty}\|\phi_n\|^2_{L^2}=\|\nu\|^2_{L^\infty},$$
since $\{\phi_n\}$ is an orthonormal basis in $L^2$.
Using the last relations and taking into account \eqref{21-1}, \eqref{W-1}, we obtain
\begin{eqnarray}\label{u_x}
\|\partial_x u(t,\cdot)\|^2_{L^2}&\lesssim&\sum\limits_{n=1}^\infty\left(\left|\sqrt{\lambda_n}A_n \right|^2+|B_n|^2\right)\left(1+\|\nu\|^2_{L^2}\right)\nonumber\\
&+&\sum\limits_{n=1}^\infty\left(|A_n|^2+\left|\frac{1}{\sqrt{\lambda_n}}B_n\right|^2\right)\|\nu\|^2_{L^\infty}\nonumber\\
&\leq& \left(1+\|\nu\|^2_{L^2}\right)\left(\|u_0\|^2_{W^1_\mathcal{L}}+\|u_1\|^2_{L^2}\right)\nonumber\\
&+&\|\nu\|^2_{L^\infty}\left(\|u_0\|^2_{L^2}+\|u_1\|^2_{W^{-1}_{\mathcal{L}}}\right),
\end{eqnarray}
implying \eqref{est3}.

Let us get next estimates by using that $\phi''_n(x)=(q(x)-\lambda_n)\phi_n(x),$
\begin{eqnarray}\label{u_xx}
\left\|\partial_x^2u(t, \cdot)\right\|^2_{L^2}&=&\int\limits_0^1\left|\partial^2_xu(t,x)\right|^2dx\nonumber\\
&=&\int\limits_0^1\left|\sum\limits_{n=1}^\infty \left[A_n \cos \left(\sqrt{\lambda_n}t\right)+\frac{1}{\sqrt{\lambda_n}} B_n \sin\left(\sqrt{\lambda_n}t\right)\right]\phi''_n(x)\right|^2dx\nonumber\\
&\lesssim& \int\limits_0^1\left(\sum\limits_{n=1}^\infty |A_n|^2|(q(x)-\lambda_n)\phi_n(x)|^2dx+\left|\frac{B_n}{\sqrt{\lambda_n}}\right|^2|(q(x)-\lambda_n)\phi_n(x)|^2\right)dx\nonumber\\
&\lesssim&\int\limits_0^1|q(x)|^2\sum\limits_{n=1}^\infty \left(|A_n|^2|\phi_n(x)|^2 +\left|\frac{B_n}{\sqrt{\lambda_n}}\right|^2|\phi_n(x)|^2\right)dx+\nonumber\\
&+&\int\limits_0^1\sum\limits_{n=1}^\infty \left(|\lambda_n A_n|^2|\phi_n(x)|^2 +\left|\lambda^\frac{1}{2}_n B_n\right|^2|\phi_n(x)|^2\right)dx\nonumber\\
&\leq&\|q\|^2_{L^\infty}\sum\limits_{n=1}^\infty \left(|A_n|^2 +\left|\frac{B_n}{\sqrt{\lambda_n}}\right|^2\right)+\sum\limits_{n=1}^\infty \left|\lambda_n A_n\right|^2 +\sum\limits_{n=1}^\infty\left|\sqrt{\lambda_n}B_n\right|^2. 
\end{eqnarray}
Using the property of the operator $\mathcal{L}$ and the Parseval identity for the last two terms in \eqref{u_xx}, we obtain
\begin{eqnarray*}
 \sum\limits_{n=1}^\infty\left|\sqrt{\lambda_n}B_n\right|^2&=& \sum\limits_{n=1}^\infty\left|\int\limits_0^1\sqrt{\lambda_n}u_1(x)\phi_n(x)dx\right|^2\\
 &=&\sum\limits_{n=1}^\infty\left|\int\limits_0^1\mathcal{L}^\frac{1}{2}u_1(x)\phi_n(x)dx\right|^2=\left\|\mathcal{L}^\frac{1}{2}u_1\right\|^2_{L^2}=\|u_1\|^2_{W^1_\mathcal{L}},
\end{eqnarray*}
also
$$\sum\limits_{n=1}^\infty \left|\lambda_n A_n\right|^2=\|u_0\|^2_{W^2_{\mathcal{L}}},$$
using the last expressions and \eqref{u_x} we finally get
\begin{eqnarray*}
\left\|\partial^2_xu(t,\cdot)\right\|^2_{L^2} 
&\lesssim &\|q\|^2_{L^\infty} \left(\|u_0\|^2_{L^2}+\|u_1\|^2_{W^{-1}_\mathcal{L}}\right)+\left\|u_0\right\|^2_{W^2_\mathcal{L}}+\|u_1\|^2_{W^1_\mathcal{L}},
\end{eqnarray*}
implying \eqref{est4}.

Let us carry out the last estimate \eqref{est5} using that $\mathcal{L}^ku=\lambda_n^ku$ and Parseval's identity:
\begin{eqnarray*}
 \left\|u(t, \cdot)\right\|^2_{W^k_\mathcal{L}}&=&\left\|\mathcal{L}^\frac{k}{2}u(t, \cdot)\right\|^2_{L^2}=\int\limits_0^1\left|\mathcal{L}^\frac{k}{2}u(t,x)\right|^2dx\\
&=&\int\limits_0^1\left|\sum\limits_{n=1}^\infty \left[A_n \cos \left(\sqrt{\lambda_n}t\right)+\frac{1}{\sqrt{\lambda_n}} B_n \sin\left(\sqrt{\lambda_n}t\right)\right]\lambda_n^\frac{k}{2}\phi_n(x)\right|^2dx\\
&\lesssim&\sum\limits_{n=1}^\infty\left( \left| \lambda_n^\frac{k}{2}A_n\right|^2+ \left| \lambda_n^\frac{k-1}{2}B_n\right|^2\right)\\
&=&\left\|\mathcal{L}^\frac{k}{2}u_0\right\|^2_{L^2}+\left\|\mathcal{L}^\frac{k-1}{2}u_1\right\|^2_{L^2}=\left\|u_0\right\|^2_{W^k_\mathcal{L}}+\left\|u_1\right\|^2_{W^{k-1}_\mathcal{L}}.
\end{eqnarray*}
The proof of Theorem \ref{th1} is complete.
\end{proof}

The following statement removes the dependence of Sobolev spaces depending on $\mathcal{L}$ by sacrificing the regularity of the data.

\begin{cor}\label{cor1}
Assume that $q,\, \nu \in L^\infty(0,1)$. If the initial data satisfy $(u_0,\, u_1) \in L^2(0,1)$ and $(u_0'', \, u''_1)\in L^2(0,1)$ then the wave equation \eqref{C.p1} with the initial/boundary problems  \eqref{C.p2}-\eqref{C.p3}  has unique solution $u\in C([0,T], L^2(0,1))$ which satisfies the estimates
\begin{equation}\label{ec1}
    \|u(t,\cdot)\|^2_{L^2}\lesssim \|u_0\|^2_{L^2}+\|u_1\|^2_{L^2},
\end{equation}
\begin{equation}\label{ec2}
\|\partial_t u(t,\cdot)\|^2_{L^2}\lesssim \|u''_0\|^2_{L^2}+\|q\|^2_{L^\infty}\|u_0\|^2_{L^2}+\|u_1\|^2_{L^2},
\end{equation}
\begin{eqnarray}\label{ec3}
\|\partial_x u(t,\cdot)\|^2_{L^2}
&\lesssim& \left(1+\|\nu\|^2_{L^2}\right)\left(\|u_0''\|^2_{L^2}+\|q\|^2_{L^\infty}\|u_0\|^2_{L^2}+\|u_1\|^2_{L^2}\right)\nonumber\\
&+&\|\nu\|^2_{L^\infty}\left(\|u_0\|^2_{L^2}+\|u_1\|^2_{L^2}\right),
\end{eqnarray}
\begin{equation}\label{ec4}
\left\|\partial_x^2u(t, \cdot)\right\|^2_{L^2}\lesssim\|q\|^2_{L^\infty}\left(\|u_0\|^2_{L^2}+\|u_1\|^2_{L^2}\right)+\|u''_0\|^2_{L^2}+\|u''_1\|^2_{L^2},
\end{equation}
where  the constants in these inequalities are independent of $u_0$, $u_1$ and $q$.
\end{cor}
\begin{proof}
By using inequality \eqref{25} we have
\begin{eqnarray}\label{26}
\|u(t, \cdot)\|^2_{L^2}&\leq& \sum\limits_{n=1}^\infty\left(\int\limits_0^1|A_n|^2|\phi_n(x)|^2dx+\int\limits_0^1\left|\frac{B_n}{\sqrt{\lambda_n}}\right|^2|\phi_n(x)|^2dx\right).
\end{eqnarray}
In Theorem \ref{th1} we obtained estimates with respect to the operator $\mathcal{L}$, but here we want to obtain estimates with respect to the initial data $(u_0,\, u_1)$ and potential $q(x)$. According to \eqref{e-val}, we can assume $\lambda_n\geq 1$, thus, we can use the estimate
\begin{equation}\label{B-lam}
\int\limits_0^1\left|\frac{B_n}{\sqrt{\lambda_n}}\right|^2|\phi_n(x)|^2dx\leq \int\limits_0^1|B_n|^2|\phi_n(x)|^2dx.
\end{equation}
Thus, by using the Parseval identity in \eqref{26} taking into account the last relation, we obtain
$$\|u(t, \cdot)\|^2_{L^2}\lesssim \sum\limits_{n=1}^\infty\left(|A_n|^2+|B_n|^2\right) = \|u_0\|^2_{L^2}+\|u_1\|^2_{L^2}.$$
By \eqref{t26} we have
\begin{equation*}
\|\partial_t u(t,\cdot)\|^2 \lesssim \int\limits_0^1\left(\sum\limits_{n=1}^\infty|\sqrt{\lambda_n} A_n |^2+\sum\limits_{n=1}^\infty|B_n|^2\right)dx.
\end{equation*}
The second term of this sum gives the norm of $\|u_1\|^2_{L^2}$ by the Parseval identity. Since $\lambda_n$ are eigenvalues of the operator $\mathcal{L}$, we obtain 
\begin{eqnarray*}\label{21}\sum\limits_{n=1}^\infty|\sqrt{\lambda_n}A_n|^2&=& \sum\limits_{n=1}^\infty\left|\sqrt{\lambda_n}\int\limits_0^1 u_0(x)\phi_n(x)dx\right|^2 \leq\sum\limits_{n=1}^\infty\left| \int\limits_0^1 \lambda_n u_0(x)\phi_n(x)dx\right|^2\nonumber\\
&=&\sum\limits_{n=1}^\infty\left| \int\limits_0^1 \left(-u''_0(x)+q(x)u_0(x)\right)\phi_n(x)dx\right|^2\\
&\lesssim& \sum\limits_{n=1}^\infty\left| \int\limits_0^1 u''_0(x)\phi_n(x)dx\right|^2+\sum\limits_{n=1}^\infty\left|\int\limits_0^1q(x)u_0(x)\phi_n(x)dx\right|^2.
      \end{eqnarray*}
Using Parseval's identity and since $q\in L^\infty$, we have
$$\sum\limits_{n=1}^\infty\left|\int\limits_0^1q(x)u_0(x)\phi_n(x)dx\right|^2=\sum\limits_{n=1}^\infty\left|\langle (q u_0),\phi_n\rangle\right|^2=\|q u_0\|^2_{L^2}\leq \|q\|^2_{L^\infty} \|u_0\|^2_{L^2},$$
and also we use Parseval's identity for the first term
$$\sum\limits_{n=1}^\infty\left| \int\limits_0^1 u''_0(x)\phi_n(x)dx\right|^2=\sum\limits_{n=1}^\infty |u_{0,n}''|^2=\|u_0''\|^2_{L^2},$$
therefore
\begin{eqnarray}\label{LA}\sum\limits_{n=1}^\infty|\sqrt{\lambda_n}A_n|^2&\lesssim& \|u_0''\|^2_{L^2}+\|q\|^2_{L^\infty}\|u_0\|^2_{L^2}.
\end{eqnarray}
Thus, 
$$\|\partial_t u(t,\cdot)\|^2_{L^2}\lesssim \|u''_0\|^2_{L^2}+\|q\|^2_{L^\infty}\|u_0\|^2_{L^2}+\|u_1\|^2_{L^2}.$$

From \eqref{u_x} we have
\begin{eqnarray*}
\|\partial_x u(t,\cdot)\|^2_{L^2}&=&\int\limits_0^1|\partial_xu(t,x)|^2dt\\
&=&
\int\limits_0^1\left|\sum\limits_{n=1}^\infty\left[A_n\cos\left(\sqrt{\lambda_n}t\right)+\frac{1}{\sqrt{\lambda_n}}B_n\sin \left(\sqrt{\lambda_n}t\right)\right]\phi'_n(x)\right|^2dx\\
&\lesssim&\sum\limits_{n=1}^\infty\left(\left|\sqrt{\lambda_n}A_n \right|^2+|B_n|^2\right)\left(1+\|\nu\|^2_{L^2}\right)\\
&+&\sum\limits_{n=1}^\infty\left(|A_n|^2+\left|\frac{1}{\sqrt{\lambda_n}}B_n\right|^2\right)\|\nu\|^2_{L^\infty}.
\end{eqnarray*}
Using \eqref{LA}, \eqref{B-lam} and Parseval's identity we get
\begin{eqnarray*}
\|\partial_x u(t,\cdot)\|^2_{L^2}&\lesssim&
\left(1+\|\nu\|^2_{L^2}\right)\left(\|u_0''\|^2_{L^2}+\|q\|^2_{L^\infty}\|u_0\|^2_{L^2}+\|u_1\|^2_{L^2}\right)\\
&+&\|\nu\|^2_{L^\infty}\left(\|u_0\|^2_{L^2}+\|u_1\|^2_{L^2}\right),
\end{eqnarray*}
implying \eqref{ec3}. From \eqref{u_xx} we have
\begin{eqnarray*}
\left\|\partial_x^2u(t, \cdot)\right\|^2_{L^2}&=&\int\limits_0^1\left|\partial^2_xu(t,x)\right|^2dx\\
&=&\int\limits_0^1\left|\sum\limits_{n=1}^\infty \left[A_n \cos \left(\sqrt{\lambda_n}t\right)+\frac{1}{\sqrt{\lambda_n}} B_n \sin\left(\sqrt{\lambda_n}t\right)\right]\phi''_n(x)\right|^2dx\\
&\lesssim& \|q\|^2_{L^\infty}\sum\limits_{n=1}^\infty \left(|A_n|^2 +\left|\frac{B_n}{\sqrt{\lambda_n}}\right|^2\right)+\sum\limits_{n=1}^\infty \left|\lambda_n A_n\right|^2 +\sum\limits_{n=1}^\infty\left|\sqrt{\lambda_n}B_n\right|^2.
\end{eqnarray*}
Using \eqref{LA} we get
$$\sum\limits_{n=1}^\infty \left|\lambda_n A_n\right|^2\lesssim \|u''_0\|^2_{L^2}+\|q\|^2_{L^\infty}\|u_0\|^2_{L^2},$$
and similarly
$$\sum\limits_{n=1}^\infty \left|\sqrt{\lambda_n} B_n\right|^2\leq \sum\limits_{n=1}^\infty \left|\lambda_n B_n\right|^2\lesssim \|u''_1\|^2_{L^2}+\|q\|^2_{L^\infty}\|u_1\|^2_{L^2}.$$
Thus, taking into account the last inequalities and \eqref{B-lam}, we obtain
$$\left\|\partial_x^2u(t, \cdot)\right\|^2_{L^2}\lesssim \|q\|^2_{L^\infty}\left(\|u_0\|^2_{L^2}+\|u_1\|^2_{L^2}\right)+\|u''_0\|^2_{L^2}+\|u''_1\|^2_{L^2}.$$
The proof of Corollary \ref{cor1} is complete.
\end{proof}

\section{Non-homogeneous equation case}

In this section, we are going to give brief ideas for how to deal with the non-homogeneous wave equation with initial/boundary conditions
\begin{equation}\label{nonh}
    \left\{\begin{array}{l}
    \partial^2_t u(t,x)+\mathcal{L} u(t,x)=f(t,x),\qquad (t,x)\in [0,T]\times (0,1),\\
    u(0,x)=u_0(x),\quad x\in (0,1),\\
    \partial_tu(0,x)=u_1(x),\quad x\in(0,1),\\
    u(t,0)=0=u(t,1),\quad t\in[0,T].
    \end{array}\right.
\end{equation}
where operator $\mathcal{L}$ is defined by 
$$\mathcal{L}=-\frac{\partial^2}{\partial x^2}+q(x),\qquad x\in(0,1).$$

\begin{thm}\label{non-hom}
Assume that $q \in L^\infty(0,1)$ and $f\in C^1([0,T],L^2(0,1))$. For any $k\in \mathbb{R}$ if the initial data satisfy $(u_0,\, u_1) \in W^{1+k}_\mathcal{L}\times W^k_\mathcal{L}$ then the non-homogeneous wave equation with initial/boundary conditions \eqref{nonh} has unique solution $u\in C([0,T], W^{1+k}_\mathcal{L})\cap C^1([0,T], W^{k}_\mathcal{L})$ which satisfies the estimates
\begin{equation}\label{es-nh1}
        \|u(t,\cdot)\|^2_{L^2}\lesssim \|u_0\|^2_{L^2}+\|u_1\|^2_{W^{-1}_\mathcal{L}}+2T^2\|f\|^2_{C([0,T],L^2(0,1))},
\end{equation}
\begin{equation}\label{es-nh2}
\|\partial_tu(t,\cdot)\|^2_{L^2}\lesssim \|u_0\|^2_{W^1_\mathcal{L}}+\|u_1\|^2_{L^2}+2T^2\|f\|^2_{C([0,T],L^2(0,1))},
\end{equation}
\begin{eqnarray}\label{es-nh3}
\|\partial_xu(t,\cdot)\|^2_{L^2}&\lesssim& \left(1+\|\nu\|^2_{L^2}\right)\left(\|u_0\|^2_{W^1_\mathcal{L}}+\|u_1\|^2_{L^2}+2T^2\|f\|^2_{C([0,T],L^2(0,1))}\right)\\
&+&\|\nu\|^2_{L^\infty}\left(\|u_0\|^2_{L^2}+\|u_1\|^2_{W^{-1}_{\mathcal{L}}}+2T^2\|f\|^2_{C([0,T],L^2(0,1))}\right),
\end{eqnarray}
\begin{eqnarray}\label{es-nh4}
\|\partial^2_xu(t,\cdot)\|^2_{L^2}&\lesssim& \|q\|^2_{L^\infty} \left(\|u_0\|^2_{L^2}+\|u_1\|^2_{W^{-1}_\mathcal{L}}+2T^2\|f\|^2_{C([0,T],L^2(0,1))}\right) \\
&+&\|u_0\|^2_{W^2_\mathcal{L}}+\|u_1\|^2_{W^1_\mathcal{L}}+2T^2\|f\|^2_{C^1([0,T],L^2(0,1))},
\end{eqnarray}
where the constants in these inequalities are independent of $u_0$, $u_1$, $q$ and $f$.
\end{thm}

\begin{proof}
We can use the eigenfunctions \eqref{sol-SL} of the corresponding (homogeneous) eigenvalue problem \eqref{4}, and look for a solution in the series form
\begin{equation}\label{nonhu}
    u(t,x)=\sum\limits_{n=1}^\infty u_n(t)\phi_n(x),
\end{equation}
where
$$u_n(t) =\int\limits_0^1 u(t,x)\phi_n(x)dx.$$
We can similarly expand the source function,
\begin{equation}\label{func}
    f(t,x)=\sum\limits_{n=1}^\infty f_n(t)\phi_n(x),\qquad f_n(t)=\int\limits_0^1f(t,x)\phi_n(x)dx.
\end{equation}

Now, since we are looking for a twice differentiable function $u(t,x)$ that satisfies the homogeneous
Dirichlet boundary conditions, we can differentiate the Fourier series \eqref{nonhu} term by term and using that the $\phi_n(x)$ satisfies the equation \eqref{4} to obtain
\begin{equation}\label{uxx}
    u_{xx}(t,x) = \sum\limits_{n=1}^\infty u_n(t)\phi''_n(x)=\sum\limits_{n=1}^\infty u_n(t)(q(x)-\lambda_n)\phi_n(x).
\end{equation}

We can also twice differentiate the series \eqref{func} with respect to $t$ to obtain
\begin{equation}\label{utt}
    u_{tt}(t,x) = \sum\limits_{n=1}^\infty u''_n(t)\phi_n(x),
\end{equation}
since the Fourier coefficients of $u_{tt}(t,x)$ are
$$\int\limits_0^1u_{tt}(t,x)\phi_n(x)dx=\frac{\partial^2}{\partial t^2}\left[\int\limits_0^1u(t,x)\phi_n(x)dx\right]=u''_n(t).$$
Differentiation under the above integral is allowed since the resulting integrand is continuous.

Substituting \eqref{utt} and \eqref{uxx} into the equation, and using \eqref{func}, we have
$$\sum\limits_{n=1}^\infty u''_n(t)\phi_n(x)-\sum\limits_{n=1}^\infty u_n(t)\left(q(x)-\lambda_n\right)\phi_n(x)+q(x)\sum\limits_{n=1}^\infty u_n(t)\phi_n(x)=\sum\limits_{n=1}^\infty f_n(t)\phi_n(x),$$
and after a slight rearrangement, we get
$$\sum\limits_{n=1}^\infty \left[u''_n(t)+\lambda_nu_n(t)\right]\phi_n(x)=\sum\limits_{n=1}^\infty f_n(t)\phi_n(x).$$
But then, due to the completeness,
$$u''_n(t)+\lambda_nu_n(t)=f_n(t), \qquad n=1,2,...,$$
which are ODEs for the coefficients $u_n(t)$ of the series \eqref{nonhu}. By the method of variation of constants we get 
\begin{eqnarray*}
u_n(t)&=&A_n\cos \left(\sqrt{\lambda_n}t\right)+\frac{1}{\sqrt{\lambda_n}}B_n \sin \left(\sqrt{\lambda_n}t\right)\\
&-&\frac{1}{\sqrt{\lambda_n}}\cos \left(\sqrt{\lambda_n}t\right)\int\limits_0^t \sin \left(\sqrt{\lambda_n}s\right)f_n(s)ds\\
&+& \frac{1}{\sqrt{\lambda_n}}\sin \left(\sqrt{\lambda_n}t\right)\int\limits_0^t \cos \left(\sqrt{\lambda_n}s\right)f_n(s)ds,
\end{eqnarray*}
where
$$A_n=\int\limits_0^1u_0(x)\phi_n(x)dx,\qquad B_n=\int\limits_0^1u_1(x)\phi_n(x)dx.$$
Thus, we can write a solution of the equation \eqref{nonh} in the form
\begin{eqnarray}\label{sol-nh}
u(t,x)&=&\sum\limits_{n=0}^\infty\left[A_n\cos\left(\sqrt{\lambda_n}t\right)+\frac{1}{\sqrt{\lambda_n}}B_n\sin \left(\sqrt{\lambda_n}t\right)\right]\phi_n(x) \nonumber\\
&-& \sum\limits_{n=1}^\infty\frac{1}{\sqrt{\lambda_n}}\cos\left(\sqrt{\lambda_n}t\right)\int\limits_0^t\sin\left(\sqrt{\lambda_n}s\right)f_n(s)ds\phi_n(x)\nonumber\\
&+&\sum\limits_{n=1}^\infty\frac{1}{\sqrt{\lambda_n}}\sin \left(\sqrt{\lambda_n}t\right)\int\limits_0^t \cos \left(\sqrt{\lambda_n}s\right)f_n(s)ds\phi_n(x).
\end{eqnarray}

Let us estimate $\|u(t,\cdot)\|_{L^2}$. For this we use the estimates
\begin{eqnarray}\label{est-nonh}
\int\limits_0^1|u(t,x)|^2dx&\lesssim&\int\limits_0^1\left|\sum\limits_{n=0}^\infty\left[A_n\cos\left(\sqrt{\lambda_n}t\right)+\frac{1}{\sqrt{\lambda_n}}B_n\sin \left(\sqrt{\lambda_n}t\right)\right]\phi_n(x)\right|^2dx \nonumber\\
&+& \int\limits_0^1\left|\sum\limits_{n=1}^\infty\frac{1}{\sqrt{\lambda_n}}\cos\left(\sqrt{\lambda_n}t\right)\int\limits_0^t\sin\left(\sqrt{\lambda_n}s\right)f_n(s)ds\phi_n(x)\right|^2dx\nonumber\\
&+&\int\limits_0^1\left|\sum\limits_{n=1}^\infty\frac{1}{\sqrt{\lambda_n}}\sin \left(\sqrt{\lambda_n}t\right)\int\limits_0^t \cos \left(\sqrt{\lambda_n}s\right)f_n(s)ds\phi_n(x)\right|^2dx\nonumber\\
&=&I_1+I_2+I_3.
\end{eqnarray}
For $I_1$ by using the \eqref{est1} for the homogeneous case we have that
$$I_1:=\int\limits_0^1\left|\sum\limits_{n=0}^\infty\left[A_n\cos\left(\sqrt{\lambda_n}t\right)+\frac{1}{\sqrt{\lambda_n}}B_n\sin \left(\sqrt{\lambda_n}t\right)\right]\phi_n(x)\right|^2dx\lesssim \|u_0\|^2_{L^2}+\|u_1\|^2_{W^{-1}_\mathcal{L}}.$$
Now we estimate $I_2$ in \eqref{est-nonh} as
$$I_2:=\int\limits_0^1\left|\sum\limits_{n=1}^\infty\frac{1}{\sqrt{\lambda_n}}\cos\left(\sqrt{\lambda_n}t\right)\int\limits_0^t\sin\left(\sqrt{\lambda_n}s\right)f_n(s)ds\phi_n(x)\right|^2dx \lesssim \sum\limits_{n=1}^\infty \left[\int\limits_0^t|f_n(s)|ds\right]^2.$$
Using Holder's inequality and taking into account that $t\in [0,T]$ we get
$$\left[\int\limits_0^t|f_n(s)|ds\right]^2\leq \left[\int\limits_0^T 1\cdot| f_n(t)|dt\right]^2\leq T\int\limits_0^T| f_n(t)|^2dt,$$
since $f_n(t)$ is the Fourier's coefficient of the function $f(t,x)$ and by Parseval's identity we obtain
$$
\sum\limits_{n=1}^\infty T\int\limits_0^T|f_n(t)|^2dt= T\int\limits_0^T\sum\limits_{n=1}^\infty|f_n(t)|^2dt = T\int\limits_0^T\|f(t,\cdot)\|^2_{L^2}dt.$$
Since
$$\|f\|_{C([0,1],L^2(0,1))}=\max\limits_{0\leq t\leq T}\|f(t,\cdot)\|_{L^2},$$
we will arrive at an inequality
$$T\int\limits_0^T\|f(t,\cdot)\|^2_{L^2}dt\leq T^2\|f\|^2_{C([0,T],L^2(0,1))}.$$
Thus,
\begin{eqnarray}\label{I2}
I_2&=&\int\limits_0^1\left|\sum\limits_{n=1}^\infty\frac{1}{\sqrt{\lambda_n}}\cos\left(\sqrt{\lambda_n}t\right)\int\limits_0^t\sin\left(\sqrt{\lambda_n}s\right)f_n(s)ds\phi_n(x)\right|^2dx\nonumber\\
&\lesssim&  T^2\|f\|^2_{C([0,T],L^2(0,1))},
\end{eqnarray}
and $I_3$ in \eqref{est-nonh} is evaluated similarly
\begin{eqnarray}\label{I3}
    I_3&:=&\int\limits_0^1\left|\sum\limits_{n=1}^\infty\frac{1}{\sqrt{\lambda_n}}\sin \left(\sqrt{\lambda_n}t\right)\int\limits_0^t \cos \left(\sqrt{\lambda_n}s\right)f_n(s)ds\phi_n(x)\right|^2dx \nonumber\\
    &\lesssim& T^2\|f\|^2_{C([0,T],L^2(0,1))}.
\end{eqnarray}

We finally get
$$
    \|u(t,\cdot)\|^2_{L^2}\lesssim \|u_0\|^2_{L^2}+\|u_1\|^2_{W^{-1}_\mathcal{L}}+2T^2\|f\|^2_{C([0,T],L^2(0,1))}.
$$
Let us estimate $\|\partial_tu(t,\cdot)\|_{L^2}$, for this we calculate $\partial_tu(t,x)$ as follows
\begin{eqnarray*}
\partial_tu(t,x)&=&\sum\limits_{n=0}^\infty\left[-\sqrt{\lambda_n}A_n\sin\left(\sqrt{\lambda_n}t\right)+B_n\cos \left(\sqrt{\lambda_n}t\right)\right]\phi_n(x) \\
&+& \sum\limits_{n=1}^\infty\sin\left(\sqrt{\lambda_n}t\right)\int\limits_0^t\sin\left(\sqrt{\lambda_n}s\right)f_n(s)ds\phi_n(x)\\
&+&
\sum\limits_{n=1}^\infty\cos \left(\sqrt{\lambda_n}t\right)\int\limits_0^t \cos \left(\sqrt{\lambda_n}s\right)f_n(s)ds\phi_n(x),
\end{eqnarray*}
then
\begin{eqnarray*}
\|\partial_tu(t,\cdot)\|^2_{L^2}&=&\int\limits_0^1|\partial_tu(t,x)|^2dx\\
&\lesssim&\int\limits_0^1\left|\sum\limits_{n=0}^\infty\left[-\sqrt{\lambda_n}A_n\sin\left(\sqrt{\lambda_n}t\right)+B_n\cos \left(\sqrt{\lambda_n}t\right)\right]\phi_n(x)\right|^2dx \\
&+& \int\limits_0^1\left|\sum\limits_{n=1}^\infty\sin\left(\sqrt{\lambda_n}t\right)\int\limits_0^t\sin\left(\sqrt{\lambda_n}s\right)f_n(s)ds\phi_n(x)\right|^2dx\\
&+&
\int\limits_0^1\left|\sum\limits_{n=1}^\infty\cos \left(\sqrt{\lambda_n}t\right)\int\limits_0^t \cos \left(\sqrt{\lambda_n}s\right)f_n(s)ds\phi_n(x)\right|^2dx\\
&\lesssim&\sum\limits_{n=0}^\infty\left|-\sqrt{\lambda_n}A_n\sin\left(\sqrt{\lambda_n}t\right)+B_n\cos \left(\sqrt{\lambda_n}t\right)\right|^2 \\
&+& \sum\limits_{n=1}^\infty\left|\sin\left(\sqrt{\lambda_n}t\right)\int\limits_0^t\sin\left(\sqrt{\lambda_n}s\right)f_n(s)ds\right|^2\\
&+&\sum\limits_{n=1}^\infty\left|\cos \left(\sqrt{\lambda_n}t\right)\int\limits_0^t \cos \left(\sqrt{\lambda_n}s\right)f_n(s)ds\right|^2,
\end{eqnarray*}
by using the \eqref{est2} for the homogeneous case and conducting evaluations as in \eqref{I2}, \eqref{I3} we obtain
$$\|\partial_tu(t,\cdot)\|^2_{L^2}\lesssim \|u_0\|^2_{W^1_\mathcal{L}}+\|u_1\|^2_{L^2}+2T^2\|f\|^2_{C([0,T],L^2(0,1))}.$$

Let us carry out next estimate
\begin{eqnarray*}
\|\partial_xu(t,\cdot)\|^2_{L^2}&=&\int\limits_0^1|\partial_xu(t,x)|^2dx\\
&\lesssim& \int\limits_0^1\left|\sum\limits_{n=0}^\infty\left[A_n\cos\left(\sqrt{\lambda_n}t\right)+\frac{1}{\sqrt{\lambda_n}}B_n\sin \left(\sqrt{\lambda_n}t\right)\right]\phi'_n(x)\right|^2dx \nonumber\\
&+& \int\limits_0^1\left|\sum\limits_{n=1}^\infty\frac{1}{\sqrt{\lambda_n}}\cos\left(\sqrt{\lambda_n}t\right)\int\limits_0^t\sin\left(\sqrt{\lambda_n}s\right)f_n(s)ds\phi'_n(x)\right|^2dx\nonumber\\
&+&\int\limits_0^1\left|\sum\limits_{n=1}^\infty\frac{1}{\sqrt{\lambda_n}}\sin \left(\sqrt{\lambda_n}t\right)\int\limits_0^t \cos \left(\sqrt{\lambda_n}s\right)f_n(s)ds\phi'_n(x)\right|^2dx.
\end{eqnarray*}
Using \eqref{est3} for the homogeneous case and estimating as in \eqref{I2}, \eqref{I3} we get
\begin{eqnarray*}
\|\partial_xu(t,\cdot)\|^2_{L^2}&\lesssim& \left(1+\|\nu\|^2_{L^2}\right)\left(\|u_0\|^2_{W^1_\mathcal{L}}+\|u_1\|^2_{L^2}+2T^2\|f\|^2_{C([0,T],L^2(0,1))}\right)\\
&+&\|\nu\|^2_{L^\infty}\left(\|u_0\|^2_{L^2}+\|u_1\|^2_{W^{-1}_{\mathcal{L}}}+2T^2\|f\|^2_{C([0,T],L^2(0,1))}\right).
\end{eqnarray*}

For the estimate $\|\partial^2_xu(t,\cdot)\|_{L^2}$ we use that $\phi_n''(x)=(q(x)-\lambda_n)\phi_n(x)$, to deduce
\begin{eqnarray*}
\|\partial^2_xu(t,\cdot)\|^2_{L^2}&=&\int\limits_0^1|\partial^2_xu(t,x)|^2dx\\
&\lesssim& \int\limits_0^1\left|\sum\limits_{n=0}^\infty\left[A_n\cos\left(\sqrt{\lambda_n}t\right)+\frac{1}{\sqrt{\lambda_n}}B_n\sin \left(\sqrt{\lambda_n}t\right)\right](q(x)-\lambda_n)\phi_n(x)\right|^2dx \\
&+& \int\limits_0^1\left|\sum\limits_{n=1}^\infty\frac{1}{\sqrt{\lambda_n}}\cos\left(\sqrt{\lambda_n}t\right)\int\limits_0^t\sin\left(\sqrt{\lambda_n}s\right)f_n(s)ds(q(x)-\lambda_n)\phi_n(x)\right|^2dx\\
&+&\int\limits_0^1\left|\sum\limits_{n=1}^\infty\frac{1}{\sqrt{\lambda_n}}\sin \left(\sqrt{\lambda_n}t\right)\int\limits_0^t \cos \left(\sqrt{\lambda_n}s\right)f_n(s)ds(q(x)-\lambda_n)\phi_n(x)\right|^2dx,
\end{eqnarray*}
and using \eqref{est4}, \eqref{I2}, \eqref{I3} we arrive at the estimates
\begin{eqnarray*}
\|\partial^2_xu(t,\cdot)\|^2_{L^2}&\lesssim& \|q\|^2_{L^\infty} \left(\|u_0\|^2_{L^2}+\|u_1\|^2_{W^{-1}_\mathcal{L}}+2T^2\|f\|^2_{C([0,T],L^2(0,1))}\right)\\
&+&\|u_0\|^2_{W^2_\mathcal{L}}+\|u_1\|^2_{W^1_\mathcal{L}}+J_1+J_2,
\end{eqnarray*}
where
$$J_1:=\int\limits_0^1\left|\sum\limits_{n=1}^\infty \sqrt{\lambda_n}\cos\left(\sqrt{\lambda_n}t\right)\int\limits_0^t \sin\left(\sqrt{\lambda_n}s\right)f_n(s)ds\phi_n(x)\right|^2dx$$
$$J_2:=\int\limits_0^1\left|\sum\limits_{n=1}^\infty \sqrt{\lambda_n}\cos\left(\sqrt{\lambda_n}t\right)\int\limits_0^t \sin\left(\sqrt{\lambda_n}s\right)f_n(s)ds\phi_n(x)\right|^2dx.$$
Let us estimate $J_1$
\begin{eqnarray*}
J_1&\lesssim& \int\limits_0^1\sum\limits_{n=1}^\infty \left|\sqrt{\lambda_n}\int\limits_0^t \sin\left(\sqrt{\lambda_n}s\right)f_n(s)ds\phi_n(x)\right|^2dx\\
&=&\sum\limits_{n=1}^\infty \left|\sqrt{\lambda_n}\int\limits_0^t \sin\left(\sqrt{\lambda_n}s\right)f_n(s)ds\right|^2\int\limits_0^1|\phi_n(x)|^2dx\\
&=&\sum\limits_{n=1}^\infty \left|\sqrt{\lambda_n}\int\limits_0^t \sin\left(\sqrt{\lambda_n}s\right)f_n(s)ds\right|^2,
\end{eqnarray*}
integrating by part, using Parseval's identity and conducting evaluations as in \eqref{I2}, we get 
\begin{eqnarray*}
J_1&\lesssim& \sum\limits_{n=1}^\infty \left|-\cos\left(\sqrt{\lambda_n}s\right)f_n(s)\bigg|_0^t+\int\limits_0^t\cos\left(\sqrt{\lambda_n}s\right)f'_n(s)ds\right|^2\\
&\lesssim& \sum\limits_{n=1}^\infty \left|\cos\left(\sqrt{\lambda_n}t\right)f_n(t)\right|^2+\sum\limits_{n=1}^\infty\left|f_n(0)\right|^2+\sum\limits_{n=1}^\infty\left|\int\limits_0^tf'_n(s)ds\right|^2\\
&\leq& \|f(t,\cdot)\|^2_{L^2}+\|f(0,\cdot)\|^2_{L^2}+T^2\|f'\|^2_{C([0,T],L^2(0,1))}\leq T^2\|f\|^2_{C^1([0,T],L^2(0,1))}.
\end{eqnarray*}
Using the respectively estimates for $J_2$ we get 
$$J_2\lesssim T^2\|f\|^2_{C^1([0,T],L^2(0,1))}.$$
Therefore,
\begin{eqnarray*}
\|\partial^2_xu(t,\cdot)\|^2_{L^2}&\lesssim& \|q\|^2_{L^\infty} \left(\|u_0\|^2_{L^2}+\|u_1\|^2_{W^{-1}_\mathcal{L}}+2T^2\|f\|^2_{C([0,T],L^2(0,1))}\right) \\
&+&\|u_0\|^2_{W^2_\mathcal{L}}+\|u_1\|^2_{W^1_\mathcal{L}}+2T^2\|f\|^2_{C^1([0,T],L^2(0,1))}.
\end{eqnarray*}
\end{proof}

\begin{cor}\label{cor2}
Assume that $q\in L^2(0,1)$ and $f\in C^1([0,T],L^2(0,1))$. If the initial data satisfy $(u_0,\, u_1) \in L^2(0,1)$ and $(u_0'',\, u''_1)\in L^2(0,1)$, then the non-homogeneous wave equation with initial/boundary conditions  \eqref{nonh} has unique solution $u\in C([0,T], L^2(0,1))$ such that
\begin{equation}\label{ec-nh1}
    \|u(t,\cdot)\|^2_{L^2}\lesssim \|u_0\|^2_{L^2}+\|u_1\|^2_{L^2}+2T^2\|f\|_{C([0,1],L^2(0,1))},
\end{equation}
\begin{equation}\label{ec-nh2}
\|\partial_t u(t,\cdot)\|^2_{L^2}\lesssim \|u''_0\|^2_{L^2}+\|q\|^2_{L^\infty}\|u_0\|^2_{L^2}+\|u_1\|^2_{L^2}+2T^2\|f\|_{C([0,1],L^2(0,1))},
\end{equation}
\begin{eqnarray}\label{ec-nh3}
\|\partial_x u(t,\cdot)\|^2_{L^2}
&\lesssim& \left(1+\|\nu\|^2_{L^2}\right)\left(\|u_0''\|^2_{L^2}+\|q\|^2_{L^\infty}\|u_0\|^2_{L^2}+\|u_1\|^2_{L^2}+2T^2\|f\|_{C([0,1],L^2(0,1))}\right)\nonumber\\
&+&\|\nu\|^2_{L^\infty}\left(\|u_0\|^2_{L^2}+\|u_1\|^2_{L^2}+2T^2\|f\|_{C([0,1],L^2(0,1))}\right),
\end{eqnarray}
\begin{eqnarray}\label{ec-nh4}
\left\|\partial_x^2u(t, \cdot)\right\|^2_{L^2}&\lesssim&\|q\|^2_{L^\infty}\left(\|u_0\|^2_{L^2}+\|u_1\|^2_{L^2}+2T^2\|f\|_{C([0,1],L^2(0,1))}\right)\nonumber\\
&+&\|u''_0\|^2_{L^2}+\|u''_1\|^2_{L^2}+2T^2\|f\|_{C([0,1],L^2(0,1))},
\end{eqnarray}
where the constants in these inequalities are independent of $u_0$, $u_1$, $q$ and $f$.
\end{cor}

The proof of Corollary \ref{cor2} immediately follows from Corollary \ref{cor1} and the proof of Theorem \ref{non-hom} .

\section{Very weak solutions}

In this section we will analyse the solutions for less regular potentials $q$. For this we will be using the notion of very weak solutions. 

Assume that the coefficient $q$ and initial data $(u_0,\, u_1)$ are the distributions on $(0,1)$. 

\begin{defn}\label{D1} (i)
A net of functions $\left(u_\varepsilon=u_\varepsilon(t,x)\right)$ is said to be $L^2$-moderate if there exist $N\in \mathbb{N}_0$ and $C>0$ such that 
$$\|u_\varepsilon\|_{L^2}\leq C \varepsilon^{-N}.$$
(ii) A net of functions $\left(q_\varepsilon=q_\varepsilon(x)\right)$ is said to be $L^\infty$-moderate if there exist $N\in \mathbb{N}_0$ and $C>0$ such that 
$$\|q_\varepsilon\|_{L^\infty}\leq C \varepsilon^{-N}.$$
\end{defn}

\begin{rem} We note that such assumptions are natural for distributional coefficients in the sense that regularisations of distributions are moderate. Precisely, by the structure theorems for distributions (see, e.g. \cite{Friedlander}), we know that distributions 
\begin{equation}\label{moder}
  \mathcal{D}'(0,1) \subset \{L^\infty(0,1) -\text{moderate families} \},  
\end{equation}
and we see from \eqref{moder}, that a solution to an initial/boundary problem may not exist in the sense of distributions, while it may exist in the set of $L^\infty$-moderate functions. 
\end{rem}

To give an example, let us take $f\in L^2(0,1)$, $f:(0,1)\to \mathbb{C}$. We introduce the function 
$$\Tilde{f}=\left\{\begin{array}{l}
    f, \text{ on }(0,1),  \\
    0,  \text{ on }\mathbb{R} \setminus (0,1),
\end{array}\right.$$
then $\Tilde{f}:\mathbb{R}\to \mathbb{C}$, and $\Tilde{f}\in \mathcal{E}'(\mathbb{R}).$

Let $\Tilde{f}_\varepsilon=\Tilde{f}*\psi_\varepsilon$ be obtained as the convolution of $\Tilde{f}$ with a Friedrich mollifier $\psi_\varepsilon$, where 
$$\psi_\varepsilon(x)=\frac{1}{\varepsilon}\psi\left(\frac{x}{\varepsilon}\right),\quad \text{for}\,\, \psi\in C^\infty_0(\mathbb{R}),\, \int \psi=1. $$
Then the regularising net $(\Tilde{f}_\varepsilon)$ is $L^p$-moderate for any $p \in [1,\infty)$, and it approximates $f$ on $(0,1)$:
$$0\leftarrow \|\Tilde{f}_\varepsilon-\Tilde{f}\|^p_{L^p(\mathbb{R})}\approx \|\Tilde{f}_\varepsilon-f\|^p_{L^p(0,1)}+\|\Tilde{f}_\varepsilon\|^p_{L^p(\mathbb{R}\setminus (0,1))}.$$
Now, let us introduce the notion of a very weak solution to the initial/boundary problem \eqref{C.p1}-\eqref{C.p3}.

\begin{defn}\label{D2}
Let $q\in \mathcal{D}'(0,1)$.  The net $(u_\varepsilon)_{\varepsilon>0}$ is said to be a very weak solution to the initial/boundary problem  \eqref{C.p1}-\eqref{C.p3} if there exists an $L^\infty$-moderate regularisation $q_\varepsilon$ of $q$ and $L^2$-moderate regularisations $u_{0,\varepsilon},\,u_{1,\varepsilon}$ of $u_0,\, u_1$, respectively, such that
\begin{equation}\label{vw1}
         \left\{\begin{array}{l}\partial^2_t u_\varepsilon(t,x)-\partial^2_x u_\varepsilon(t,x)+q_\varepsilon(x) u_\varepsilon(t,x)=0,\quad (t,x)\in [0,T]\times(0,1),\\
 u_\varepsilon(0,x)=u_{0,\varepsilon}(x),\,\,\, x\in (0,1), \\
\partial_t u_\varepsilon(0,x)=u_{1,\varepsilon}(x), \,\,\, x\in (0,1),\\
u_\varepsilon(t,0)=0=u_\varepsilon(t,1), \quad t\in[0,T],
\end{array}\right.\end{equation}
and $(u_\varepsilon)$ and $(\partial_t u_\varepsilon)$ are $L^{2}$-moderate.
\end{defn}

Then we have the following properties of very weak solutions.

\begin{thm}[Existence]\label{Ext}
Let the coefficient $q$ and initial data $(u_0,\, u_1)$ be distributions in $(0,1)$. Then the initial/boundary problem  \eqref{C.p1}-\eqref{C.p3} has a very weak solution.
\end{thm}
\begin{proof}
Since the formulation of \eqref{C.p1}-\eqref{C.p3} in this case might be impossible in the distributional sense due to issues related to the product of distributions, we replace \eqref{C.p1}-\eqref{C.p3} with a regularised equation. In other words, we regularize $q$, $u_0$ and $u_1$ by some corresponding sets $q_\varepsilon$, $u_{0,\varepsilon}$ and $u_{1,\varepsilon}$ of smooth functions from $ C^ \infty(0,1)$.

Hence, $q_\varepsilon$, $u_{0,\varepsilon}$ and $u_{1,\varepsilon}$ are $L^\infty$ and $L^2$-moderate regularisations of the coefficient $q$ and the Cauchy data $(u_0,u_1)$ respectively. So by Definition 1 there exist $N\in \mathbb{N}_0$ and $C_1>0,\,C_2>0,\,C_3>0$, such that
$$\|q_\varepsilon\|_{L^\infty}\leq C_1\varepsilon^{-N},\quad \|u_{0,\varepsilon}\|_{L^2}\leq C_2\varepsilon^{-N}, \quad \|u_{1,\varepsilon}\|_{L^2}\leq C_3\varepsilon^{-N}.$$

Now we fix $\varepsilon\in (0,1]$, and consider the regularised problem \eqref{vw1}. Then all discussions and calculations of Theorem \ref{th1} are valid. Thus, by Theorem \ref{th1}, the equation \eqref{vw1} has unique solution $u_\varepsilon(t,x)$ in the space $C^0([0,T];C^\infty(0,1))\cap C^1([0,T];C^\infty(0,1))$.

By Corollary \ref{cor1} and there exist $N\in \mathbb{N}_0$ and $C>0$, such that
$$\|u_\varepsilon(t,\cdot)\|_{L^2}\lesssim \|u_{0,\varepsilon}\|_{L^2}+\|u_{1,\varepsilon}\|_{L^2}\leq C\varepsilon^{-N},$$
$$\|\partial_t u_\varepsilon(t,\cdot)\|_{L^2}\lesssim \|u''_{0,\varepsilon}\|_{L^2}+\|q_\varepsilon\|_{L^\infty}\|u_{0, \varepsilon}\|_{L^2}+\|u_{1,\varepsilon}\|_{L^2}\leq C\varepsilon^{-N},$$
where the constants in these inequalities are independent of $u_0$, $u_1$, $q$.
Hence, $(u_\varepsilon)$ is moderate, and the proof of Theorem \ref{Ext} is complete.
\end{proof}

Describing the uniqueness of the very weak solutions amounts to “measuring” the changes of involved associated nets: negligibility conditions for nets of functions/distributions read as follows:

\begin{defn}[Negligibility]\label{D3}
Let $(u_\varepsilon)$, $(\Tilde{u}_\varepsilon)$ be two nets in $C^\infty(0,1)$. Then, the net $(u_\varepsilon-\Tilde{u}_\varepsilon)$ is called $L^2$-negligible, if for every $N\in \mathbb{N}$ there exist $C>0$ such that the following condition is satisfied
$$\|u_\varepsilon-\Tilde{u}_\varepsilon\|_{L^2}\leq C \varepsilon^N,$$
for all $\varepsilon\in (0,1]$. In the case where $u_\varepsilon=u_\varepsilon(t,x)$ is a net depending on $t\in [0,T]$, then the negligibility condition can be described as 
$$\|u_\varepsilon(t,\cdot)-\Tilde{u}_\varepsilon(t,\cdot)\|_{L^2}\leq C \varepsilon^N,$$
uniformly in $t\in [0,T]$. The constant $C$ can depends on $N$ but not on $\varepsilon$.
\end{defn}

Let us state the  \textquotedblleft$\varepsilon$-parameterised problems\textquotedblright \,to be considered:
\begin{equation}\label{un1}
    \left\{\begin{array}{l}\partial^2_t u_\varepsilon(t,x)-\partial^2_x u_\varepsilon(t,x)+q_\varepsilon(x) u_\varepsilon(t,x)=0,\quad (t,x)\in [0,T]\times(0,1),\\
 u_\varepsilon(0,x)=u_{0,\varepsilon}(x),\,\,\, x\in (0,1), \\
\partial_t u_\varepsilon(0,x)=u_{1,\varepsilon}(x), \,\,\, x\in (0,1),\\
u_\varepsilon(t,0)=0=u_\varepsilon(t,1),\quad t\in[0,T],
\end{array}\right.
\end{equation}
and
\begin{equation}\label{un2}
    \left\{\begin{array}{l}\partial^2_t \Tilde{u}_\varepsilon(t,x)-\partial^2_x \Tilde{u}_\varepsilon(t,x)+\Tilde{q}_\varepsilon(x) \Tilde{u}_\varepsilon(t,x)=0,\quad (t,x)\in [0,T]\times(0,1),\\
 \Tilde{u}_\varepsilon(0,x)=\Tilde{u}_{0,\varepsilon}(x),\,\,\, x\in (0,1), \\
\partial_t \Tilde{u}_\varepsilon(0,x)=\Tilde{u}_{1,\varepsilon}(x), \,\,\, x\in (0,1),\\
\Tilde{u}_\varepsilon(t,0)=0=\Tilde{u}_\varepsilon(t,1), \quad t\in[0,T].
\end{array}\right.
\end{equation}

\begin{defn}[Uniqueness of the very weak solution]\label{D4}
Let $q\in \mathcal{D}'(0,1)$. We say that initial/boundary problem \eqref{C.p1}-\eqref{C.p3} has an unique very weak solution, if
for all $L^\infty$-moderate nets $q_\varepsilon$, $\Tilde{q}_\varepsilon$, such that $(q_\varepsilon-\Tilde{q}_\varepsilon)$ is $L^\infty$-negligible; and for all $L^2$-moderate regularisations $u_{0,\varepsilon},\,\Tilde{u}_{0,\varepsilon}$, $u_{1,\varepsilon},\,\Tilde{u}_{1,\varepsilon}$ such that $(u_{0,\varepsilon}-\Tilde{u}_{0,\varepsilon})$, $(u_{1,\varepsilon}-\Tilde{u}_{1,\varepsilon})$ are $L^2$-negligible, we have that $u_\varepsilon-\Tilde{u}_\varepsilon$ is $L^2$-negligible.
\end{defn}

\begin{thm}[Uniqueness of the very weak solution]\label{Th-U}
Let the coefficient $q$ and initial data $(u_0,\, u_1)$ be distributions in $(0,1)$. Then the very weak solution to the initial/boundary problem  \eqref{C.p1}-\eqref{C.p3} is unique.
\end{thm}
\begin{proof}
We denote by $u_\varepsilon$ and $\Tilde{u}_\varepsilon$ the families of solutions to the initial/boundary problems \eqref{un1} and \eqref{un2}, respectively. Setting $U_\varepsilon$ to be the difference of these nets $U_\varepsilon:=u_\varepsilon(t,\cdot)-\Tilde{u}_\varepsilon(t,\cdot)$, then $U_\varepsilon$ solves
    \begin{equation}\label{unq}
    \left\{\begin{array}{l}\partial^2_t U_\varepsilon(t,x)-\partial^2_x U_\varepsilon(t,x)+q_\varepsilon(x) U_\varepsilon(t,x)=f_\varepsilon(t,x),\quad (t,x)\in [0,T]\times(0,1),\\
 U_\varepsilon(0,x)=(u_{0,\varepsilon}-\Tilde{u}_{0,\varepsilon})(x),\,\,\, x\in (0,1), \\
\partial_t U_\varepsilon(0,x)=(u_{1,\varepsilon}-\Tilde{u}_{1,\varepsilon})(x), \,\,\, x\in (0,1),\\
U_\varepsilon(t,0)=0=U_\varepsilon(t,1),
\end{array}\right.
\end{equation}
where we set $f_\varepsilon(t,x):=(\Tilde{q}_\varepsilon(x)-q_\varepsilon(x))\Tilde{u}_\varepsilon(t,x)$ for the mass term to the non-homogeneous initial/boundary problem \eqref{unq}.

Passing to the $L^2$-norm of the $U_\varepsilon$, by using \eqref{ec-nh1} we obtain 
$$
\|U_\varepsilon(t,\cdot)\|^2_{L^2}\lesssim \|U_\varepsilon(0,\cdot)\|^2_{L^2}+\|\partial_tU_\varepsilon(0,\cdot)\|^2_{L^2}+2T^2\|f_\varepsilon\|^2_{C([0,T],L^2(0,1))}.
$$
Since
$$\|f_\varepsilon\|^2_{C([0,T],L^2(0,1))}\leq \|q_\varepsilon-\Tilde{q}_\varepsilon\|^2_{L^\infty}\|\Tilde{u}_\varepsilon\|^2_{C([0,T],L^2(0,1))},$$
and using initial data of \eqref{unq} we get
$$
\|U_\varepsilon(t,\cdot)\|^2_{L^2}\lesssim \|u_{0,\varepsilon}-\Tilde{u}_{0,\varepsilon}\|^2_{L^2}+\|u_{1,\varepsilon}-\Tilde{u}_{1,\varepsilon}\|^2_{L^2}+2T^2\|q_\varepsilon-\Tilde{q}_\varepsilon\|^2_{L^\infty}\|\Tilde{u}_\varepsilon\|^2_{C([0,T],L^2(0,1))}.
$$
Taking into account the negligibility of the nets $u_{0,\varepsilon}-\Tilde{u}_{0,\varepsilon}$, $u_{1,\varepsilon}-\Tilde{u}_{1,\varepsilon}$ and $q_\varepsilon-\Tilde{q}_\varepsilon$ we get
$$\|U_\varepsilon(t,\cdot)\|^2_{L^2}\leq C_1\varepsilon^{N_1}+C_2\varepsilon^{N_2}+C_3\varepsilon^{N_3}\varepsilon^{-N_4}$$
for some $C_1>0,\,C_2>0,\,C_3>0,\,N_4\in \mathbb{N}$ and all $N_1,\,N_2,\,N_3\in \mathbb{N}$, since $\Tilde{u}_\varepsilon$ is moderate. Then, for some $C_M>0$ and all $M\in \mathbb{N}$
$$\|U_\varepsilon(t,\cdot)\|_{L^2}\leq C_M \varepsilon^M.$$
The last estimate holds true uniformly in $t$ , and this completes the proof of Theorem \ref{Th-U}.
\end{proof}

\begin{thm}[Consistency]\label{Th-C} Assume that $q\in L^\infty(0,1)$, and let $(q_\varepsilon)$ be any $L^\infty$-regularisation of $q$, that is $\|q_\varepsilon-q\|_{L^\infty}\to 0$ as $\varepsilon\to 0$. Let the initial data satisfy $(u_0,\, u_1) \in L^2(0,1)\times L^2(0,1)$. Let $u$ be a very weak solution of the initial/boundary problem \eqref{C.p1}-\eqref{C.p3}. Then for any families $q_\varepsilon$, $u_{0,\varepsilon}$, $u_{1,\varepsilon}$ such that $\|u_{0}-u_{0,\varepsilon}\|_{L^2}\to 0$, $\|u_{1}-u_{1,\varepsilon}\|_{L^2}\to 0$, $\|q-q_{\varepsilon}\|_{L^\infty}\to 0$ as $\varepsilon\to 0$, any representative $(u_\varepsilon)$ of the very weak solution converges as 
$$\sup\limits_{0\leq t\leq T}\|u(t,\cdot)-u_\varepsilon(t,\cdot)\|_{L^2}\to 0$$ 
for $\varepsilon\to 0$ to the unique classical solution $u\in C([0,T];L^2(0,1))$ of the initial/boundary problem \eqref{C.p1}-\eqref{C.p3} given by Theorem \ref{th1}.
\end{thm}

\begin{proof}
For $u$ and for $u_\varepsilon$, as in our assumption, we introduce an auxiliary notation $V_\varepsilon(t, x):= u(t,x)-u_\varepsilon(t,x)$. Then the net $V_\varepsilon$ is a solution to the initial/boundary problem
\begin{equation}\label{51}
    \left\{\begin{array}{l}
        \partial^2_tV_\varepsilon(t,x)-\partial^2_xV_\varepsilon(t,x)+q_\varepsilon(x)V_\varepsilon(t,x)=f_\varepsilon(t,x),\\
        V_\varepsilon(0,x)=(u_0-u_{0,\varepsilon})(x),\quad x\in (0,1),\\
        \partial_tV_\varepsilon(0,x)=(u_1-u_{1,\varepsilon})(x),\quad x\in (0,1),\\
        V_\varepsilon(t,0)=0=V_\varepsilon(t,1), \quad t\in[0,T],
    \end{array}\right.
\end{equation}
where $f_\varepsilon(t,x)=(q_\varepsilon(x)-q(x))u(t,x)$. Analogously to Theorem \ref{Th-U} we have that
$$\|V_\varepsilon(t,\cdot)\|^2_{L^2}\lesssim  \|u_{0}-{u}_{0,\varepsilon}\|^2_{L^2}+\|u_{1}-{u}_{1,\varepsilon}\|^2_{L^2}+2T^2\|q_\varepsilon-q\|^2_{L^\infty}\|u\|^2_{C([0,T],L^2(0,1))}.$$
Since
$$\|u_{0}-{u}_{0,\varepsilon}\|_{L^2}\to 0,\quad \|u_{1}-{u}_{1,\varepsilon}\|_{L^2}\to 0,\quad \|q_\varepsilon-q\|_{L^\infty}\to 0$$
for $\varepsilon\to 0$ and $u$ is a very weak solution of the initial/boundary problem \eqref{C.p1}-\eqref{C.p3} we get
$$\|V_\varepsilon(t,\cdot)\|_{L^2}\to 0$$
for $\varepsilon\to 0$. This proves Theorem \ref{Th-C}.

\end{proof}

\end{document}